\documentclass{article}

\usepackage[a4paper,margin=2cm]{geometry}
\usepackage{epic}
\usepackage{amsmath, amscd, amsfonts, amsthm, amssymb, enumerate, graphicx}
\usepackage{url}
\usepackage[matrix,arrow]{xy}

\newtheorem{thm}{Theorem}[section]
\newtheorem{lem}[thm]{Lemma}
\newtheorem{prop}[thm]{Proposition}

\theoremstyle{definition}

\newtheorem{defn}[thm]{Definition}
\newtheorem{examp}[thm]{Example}
\theoremstyle{remark}
\newtheorem*{rem}{Remark}

\DeclareMathOperator{\tr}{tr}

\DeclareMathOperator{\sgn}{sgn}

\DeclareMathOperator{\cond}{cond}
\DeclareMathOperator{\disc}{disc}

\DeclareMathOperator{\Hom}{Hom}

\DeclareMathOperator{\Disc}{Disc}

\DeclareMathOperator{\Frob}{Frob}
\DeclareMathOperator{\Pic}{Pic}
\DeclareMathOperator{\Cl}{Cl}

\DeclareMathOperator{\Aut}{Aut}
\DeclareMathOperator{\Gal}{Gal}

\newcommand{\Discs}{\mathcal{D}iscs}
\newcommand{\<}{\left\langle}
\renewcommand{\>}{\right\rangle}

\newcommand{\GL}{\mathrm{GL}}
\newcommand{\SL}{\mathrm{SL}}
\newcommand{\CC}{\mathbb{C}}
\newcommand{\FF}{\mathbb{F}}

\newcommand{\QQ}{\mathbb{Q}}
\newcommand{\PP}{\mathbb{P}}

\newcommand{\ZZ}{\mathbb{Z}}
\renewcommand{\aa}{\mathfrak{a}}

\newcommand{\dd}{\mathfrak{d}}
\newcommand{\pp}{\mathfrak{p}}

\newcommand{\OO}{\mathcal{O}}

\newcommand{\cross}{\times}
\newcommand{\tensor}{\otimes}

\newcommand{\zmat}{$\ZZ$-mat}
\newcommand{\textand}{\quad \text{and} \quad}

\renewcommand{\to}{\mathop{\rightarrow}\limits}
\newcommand{\intsec}{\cap}
\newcommand{\union}{\cup}
\newcommand{\ignore}[1]{}

\newcommand{\bbq}[8]{
\begin{minipage}{0.1\linewidth}
\xymatrix@!0{
& #5 \ar@{-}[rr]\ar@{-}[dd]
& & #6 \ar@{-}[dd]
\\
#1 \ar@{-}[ur]\ar@{-}[rr]\ar@{-}[dd]
& & #2 \ar@{-}[ur]\ar@{-}[dd]
\\
& #7 \ar@{-}[rr]
& & #8
\\
#3 \ar@{-}[rr]\ar@{-}[ur]
& & #4 \ar@{-}[ur]
}
\end{minipage}
}

\newdir^{ (}{{}*!/-5pt/@^{(}}
\newdir_{ (}{{}*!/-5pt/@_{(}}

\begin{document}

\title{A remarkable identity in class numbers of cubic rings}
\author{Evan O'Dorney}
\maketitle

\begin{abstract}
In 1997, Y.~Ohno empirically stumbled on an astoundingly simple identity relating the number of cubic rings $h(\Delta)$ of a given discriminant $\Delta$, over the integers, to the number of cubic rings $\hat h(\Delta)$ of discriminant $-27\Delta$ in which every element has trace divisible by $3$:
\begin{equation}\label{eq:abs main}
  \hat{h}(\Delta) = \begin{cases}
    3h(\Delta) & \text{if } \Delta > 0 \\
    h(\Delta) & \text{if } \Delta < 0,
  \end{cases}
\end{equation}
where in each case, rings are weighted by the reciprocal of their number of automorphisms. This allows the functional equations governing the analytic continuation of the Shintani zeta functions (the Dirichlet series built from the functions $h$ and $\hat h$) to be put in self-reflective form. In 1998, J.~Nakagawa verified \eqref{eq:abs main}. We present a new proof of \eqref{eq:abs main} that uses the main ingredients of Nakagawa's proof (binary cubic forms, recursions, and class field theory), as well as one of Bhargava's celebrated higher composition laws, while aiming to stay true to the stark elegance of the identity.
\end{abstract}

\section{Introduction}
\label{sec:intro}

Great progress has been made in recent years \cite{FThorne,BV} in analyzing statistics pertaining to cubic fields, ordered by discriminant. A basic analytic tool at one's disposal is the \emph{Shintani zeta functions}, a pair of Dirichlet series that encode the number of cubic \emph{rings} over $\ZZ$ of each nonzero discriminant:
\begin{align*}
  \zeta^+(s) &= \sum_{\substack{C/\ZZ\text{ cubic,} \\ \Disc C > 0}} \frac{(\Disc C)^{-s}}{\lvert\Aut C\rvert}
  \\
  \zeta^-(s) &= \sum_{\substack{C/\ZZ\text{ cubic,} \\ \Disc C < 0}} \frac{(-\Disc C)^{-s}}{\lvert\Aut C\rvert}.
\end{align*}
The division by the number of automorphisms is a standard trick in this discipline which ensures, among other things, that the relative weights of a ring and its subrings (some of which may be isomorphic) are in the proper ratio. Because almost all cubic fields (and rings) have trivial automorphism group, this factor has no effect in most analytic applications.

The Shintani zeta functions were introduced in 1972 by Shintani, who proved that they have meromorphic continuations to the whole complex plane satisfying a reflection formula of the form (see \cite{Nakagawa}, eq.~(0.1))
\begin{equation} \label{eq:fnl eq}
  \begin{bmatrix}
    \zeta^+(1-s) \\ \zeta^-(1-s)
  \end{bmatrix}
  = \begin{bmatrix}
    c_1(s) & c_2(s) \\ c_3(s) & c_4(s)
  \end{bmatrix}
  \begin{bmatrix}
    \hat\zeta^+(s) \\ \hat\zeta^-(s)
  \end{bmatrix}
\end{equation}
connecting them to two other Dirichlet series $\hat\zeta^+$ and $\hat\zeta^-$ (the $c_i$, which are certain elementary expressions involving the $\Gamma$ function, need not detain us). The functions $\hat\zeta^+$ and $\hat\zeta^-$ arise as follows. Call a cubic ring \emph{integer-matrix}, or \emph{$\ZZ$-mat} for short, if the trace of each of its elements is a multiple of $3$. (This name will be demystified in the next section.) The discriminant of such a ring is always divisible by $27$, making the scaling of the following Dirichlet series natural:
\begin{align*}
  \hat\zeta^+(s) &= 3^{3s} \sum_{\substack{C/\ZZ\text{ $\ZZ$-mat,} \\ \Disc C > 0}} \frac{(\Disc C)^{-s}}{\lvert\Aut C\rvert}
  \\
  \hat\zeta^-(s) &= 3^{3s} \sum_{\substack{C/\ZZ\text{ $\ZZ$-mat,} \\ \Disc C < 0}} \frac{(-\Disc C)^{-s}}{\lvert\Aut C\rvert}.
\end{align*}

Shintani's functional equation stood unimproved until 1997, when Y.~Ohno computed the first $200$ terms of each of the four zeta functions and conjectured that they are equal in pairs, up to a curiously sign-dependent scale factor:
\[
  \hat\zeta^+(s) = \zeta^-(s) \quad \text{and} \quad \hat\zeta^-(s) = 3\zeta^+(s).
\]
This implies that the Shintani zeta functions satisfy a self-reflective functional equation, just like the Riemann zeta function. This striking conjecture was verified by Nakagawa the following year. In purely algebraic form, it is the following, which will be the subject of this essay.

\begin{thm} \label{thm:main}
Let $h(\Delta)$ denote the number of cubic rings of discriminant $\Delta$, each weighted by the reciprocal of its number of automorphisms. Let $\hat{h}(\Delta)$ denote the number of \zmat{} cubic rings of discriminant $-27\Delta$, weighted in the same manner. Then for each integer $\Delta \neq 0$,
\begin{equation}\label{eq:main}
  \hat{h}(\Delta) = \begin{cases}
    3h(\Delta) & \text{if } \Delta > 0 \\
    h(\Delta) & \text{if } \Delta < 0.
  \end{cases}
\end{equation}
\end{thm}

Developments in number theory since 1998, specifically Bhargava's beautiful work in higher composition laws in the early 2000's, suggest revisiting this beautiful identity. (A \emph{higher composition law,} in Bhargava's parlance, is a parametrization of interesting algebraic objects by the orbits of an algebraic group action \cite{icm}; it need not be a group operation.) In particular, one of the main steps in Nakagawa's proof relates $\ZZ$-mat rings of discriminant $-27\Delta$ to ideals in orders of the quadratic algebra $\QQ(\sqrt{\Delta})$, and one of Bhargava's higher composition laws relates the same sort of objects. Can Bhargava's result be adapted as a replacement for Nakagawa's somewhat ad hoc computation? We answer this question affirmatively. We also find a simple recursive formula for $h(\Delta)$ or $\hat{h}(\Delta)$ valid when $\Delta$ has high prime power divisors (Theorem \ref{thm:recursion}). Finally, unlike Nakagawa, we treat the cases $\Delta > 0$ and $\Delta < 0$ simultaneously, enabling us to explain the factor of $3$ in the statement quite readily. It arises from the existence of a fundamental unit in $\QQ(\sqrt{\Delta})$, except when $\Delta$ is a square, in which case it arises from the extra automorphism of order $3$ belonging to cubic fields of square discriminant.

\begin{examp}
The simplest case of Theorem \ref{thm:main} is when $\Delta = 1$. There is just one cubic ring of discriminant $1$, namely $\ZZ \cross \ZZ \cross \ZZ$, and it has six automorphisms, so $h(1) = 1/6$. There is also just one $\ZZ$-mat ring of discriminant $-27$, namely $\ZZ[t]/(t^3-1)$, and it has a single nontrivial automorphism $t \mapsto t^2$, so $\hat h(1) = 1/2$, in accordance with the theorem.
\end{examp}

\subsection{Outline of the proof}
Our proof of Theorem \ref{thm:main} follows four main steps:
\begin{enumerate}
  \item \label{st:recur} Construct a recursion allowing one to reduce to the case where the prime powers dividing $\Delta$ are not too high (Section \ref{sec:recursions}).
  \item \label{st:Bha} Use Bhargava's theory of higher composition laws to relate cubic rings of discriminant $-27\Delta$ to certain ideals in orders of $\QQ(\sqrt{\Delta})$ (Section \ref{sec:Bhargava}).
  \item \label{st:CFT} Use class field theory to relate cubic fields of discriminant $\Delta$ to certain characters on the ideal group of the quadratic algebra $\QQ(\sqrt{\Delta})$ (Section \ref{sec:CFT}).
  \item \label{st:cubf} Combine the foregoing steps to prove the theorem (Section \ref{sec:finish}).
\end{enumerate}
The first three steps are completely independent, and we have chosen to order them in a manner that places the non-elementary material last. Each of the steps culminates in a theorem that has an analogue in Nakagawa's proof, though potentially with some conditions altered, or, in the case of step 1, a beautiful and apparently new recursive formula for $h(\Delta)$ and $\hat{h}(\Delta)$.

At first glance, the two sides of \eqref{eq:main} are analogous, even ``dual'' to each other. Indeed, the space $\QQ^4$ of rational binary cubic forms has a natural $\SL_2\QQ$-invariant skew form $aa' - \frac{1}{3}bb' + \frac{1}{3}cc' - dd'$, with respect to which the lattices of integral and \zmat{} cubic forms are mutually dual, and this duality was used by Shintani to establish the functional equation \eqref{eq:fnl eq} in \cite{Shintani}. By contrast, $h$ and $\hat h$ are treated asymmetrically in Nakagawa's proof and even more asymmetrically in the present one: we only apply class field theory to $h$ and Bhargava's parametrizations to $\hat h$, allowing us to minimize the amount of time spent treating the prime $3$ specially.

\section{Basic notions}
\label{sec:basic}

Let $A$ be a principal ideal domain (PID); quintessentially $A = \ZZ$, although we will also use $A = \ZZ_p$ in this paper. By an \emph{$n$-ic ring} over $A$ we will mean a commutative ring $C$ with unit which is isomorphic to $A^n$ as an $A$-module. Only quadratic ($n=2$) and cubic ($n=3$) rings concern us here.

The \emph{discriminant} $\Disc C$ of an $n$-ic ring is, as usual, the determinant of the trace pairing matrix $[\tr \alpha_i\alpha_j]_{i,j = 1}^n$, where $[\alpha_1, \cdots, \alpha_n]$ is any $A$-basis for $C$. It is well defined up to multiplication by the square of a unit in $A$, so if $A = \ZZ$, the discriminant is simply an integer, while if $A = \ZZ_p$, a discriminant is determined up to a finite list of possibilities by its valuation $v_p(\Disc C)$. A ring $C$ is called \emph{nondegenerate} if its discriminant is nonzero.

A classical theorem due to Stickelberger states that the discriminant of a number field, and hence of any finite-rank ring over $\ZZ$, is congruent to $0$ or $1$ modulo $4$. In the case of a cubic ring, we will soon give a direct proof. We mention Stickelberger's theorem here only to motivate the following definitions. Let
\[
  \Discs = (4\ZZ \setminus 0) \union (1 + 4\ZZ)
\]
be the set of all possible discriminants for a nondegenerate $\ZZ$-algebra. Note that there is exactly one quadratic $\ZZ$-algebra of each discriminant $\Delta \in \Discs$; we denote it by $\OO_\Delta$. Call $\Delta \in \Discs$ a \emph{fundamental discriminant} if $\Delta$ is not of the form $\Delta'k^2$, where $k > 1$ and $\Delta' \in \Discs$. The fundamental discriminants are exactly those $\Delta$ such that $\OO_\Delta$ is maximal (being either $\ZZ \cross \ZZ$ or the ring of integers of a quadratic field). A general $\Delta \in \Discs$ can be written uniquely in the form $\Delta_0 f^2$, where $f \geq 1$ and $\Delta_0$ is fundamental; we have an identification $\OO_\Delta \cong \ZZ + f\OO_{\Delta_0}$. 

Analogously, let $\Discs_p$ be the set of all possible discriminants for a nondegenerate $\ZZ_p$-algebra, namely
\[
  \Discs_p = \begin{cases}
    (4\ZZ_2 \setminus 0) \union (1 + 4\ZZ_2) & p = 2 \\
    \ZZ_p \setminus 0 & \text{otherwise.}
  \end{cases}
\]
Call $\Delta \in \Discs_p$ \emph{fundamental} if it is not $p^2$ times an element of $\Discs_p$. This is the same as requiring that the unique quadratic ring over $\ZZ_p$ of discriminant $\Delta$ be maximal. One computes that the fundamental $p$-adic discriminants are, for $p \neq 2$, those not divisible by $p^2$, and for $p = 2$, those congruent to $1$ (mod $4$) or to $8$ or $12$ (mod $16$).

If $K$ is a nondegenerate $\QQ$-algebra and $\OO_K$ is the integral closure of $\ZZ$ in $K$, then the \emph{splitting type} of a prime $p$ is the symbol $f_1^{e_1}\cdots f_r^{e_r},$ where the $f_i$ and $e_i$ are the degrees and ramification indices of the primes into which $p$ splits in $\OO_K$, or equivalently of the extensions of $\QQ_p$ into which the completed algebra $K_p$ splits. The splitting type may be defined uniformly regardless of whether $K$ itself is a field.

\section{Cubic rings and binary cubic forms}
\label{sec:rings forms}

The simplest means of studying cubic rings uses a very elementary parametrization by binary cubic forms. This parametrization was first stated over an arbitrary PID by Gross and Lucianovic (\cite{cubquat}), but the gist of it is quite old. It is often attributed to Delone and Faddeev \cite{DF}, but Delone and Faddeev themselves attribute the result to a 1914 paper of F.~W.~Levi in the preface to their book, and we will call it the Levi form in his honor. Bhargava (\cite{B2}, pp.~868--869) discovered an attractive coordinate-free formulation which we follow here.

\begin{thm} [\cite{cubquat}, Prop{.} 2.1] \label{thm:forms and rings}
Let $A$ be a PID. The association of a cubic ring $C$ to the map
\[
  \xi \mapsto 1 \wedge \xi \wedge \xi^2 : \quad C/A \to \Lambda^3 C
\]
defines a bijection between isomorphism classes of cubic rings over $A$ and orbits of binary cubic forms
\[
  \phi(x,y) = a x^3 + b x^2 y + c x y^2 + d y^3, \quad a,b,c,d \in A
\]
under the $\GL_2 A$-action
\begin{equation}\label{eq:twisted action}
  \left(\begin{bmatrix}
    p & q \\ r & s
  \end{bmatrix}
  .\ \phi\right)(x,y) = \frac{1}{p s - q r} \phi(px + ry, qx + sy).
\end{equation}
Moreover, the $A$-algebra automorphism group of $C$ is isomorphic to the stabilizer in $\GL_2 A$ of the corresponding form.
\end{thm}
\begin{proof}
Note that for $\xi \in C$ and $n \in A$, we formally have
\[
  1 \wedge (\xi+n) \wedge (\xi+n)^2 = 1 \wedge \xi \wedge \xi^2,
\]
so $\phi_C(x) = 1 \wedge \xi \wedge \xi^2$ really does define a cubic map from $C/A$ to $\Lambda^3 C$. If we pick a basis $[\bar{\alpha}, \bar{\beta}]$ for $C/A$ lifting to some basis $[1,\alpha,\beta]$ of $C$, then $\Lambda^3 C$ acquires a distinguished generator $1 \wedge \alpha \wedge \beta$ and $\phi : A^2 \to A$ becomes a cubic form.

As a preliminary claim, let us show that every cubic form $\phi$ arises from exactly one cubic ring $C$ with distinguished basis $[\bar{\alpha}, \bar{\beta}]$ for $C/A$ in this way. First note that the selection of basis $[\bar{\alpha}, \bar{\beta}]$ is tantamount to a selection of a \emph{normal basis} for $C$, that is, a basis $[1,\alpha,\beta]$ such that $\alpha\beta \in A$: if $\alpha'$, $\beta'$ are any lifts of $\bar{\alpha}$ and $\bar{\beta}$, then
\[
  \alpha'\beta' = t + u\alpha' + v\beta' \quad (t,u,v \in A),
\]
and $[1,\alpha' - v, \beta' - u]$ is the unique such basis.

Now write the multiplication table of $C$, still undetermined, in terms of this basis:
\begin{align*}
  \alpha^2 &= \ell - b\alpha + a\beta \\
  \alpha\beta &= m \\
  \beta^2 &= n - d\alpha + c\beta,
\end{align*}
where the signs and letters will be motivated momentarily. We compute
\begin{align*}
  \phi(x,y) &= 1 \wedge (\alpha x + \beta y) \wedge (\alpha x + \beta y)^2 \\
  &= 1 \wedge (\alpha x + \beta y) \wedge [(\ell - b\alpha + a\beta)x^2 + mxy + (n - d\alpha + c\beta)] \\
  &= (ax^3 + bx^2y + cxy^2 + dy^3)(1 \wedge \alpha \wedge \beta).
\end{align*}
So the cubic form $\phi$ exactly carries the information of the four coefficients $a$, $b$, $c$, and $d$.
Expanding out the associative laws $(\alpha^2)\beta = \alpha(\alpha\beta)$ and $(\alpha\beta)\beta = \alpha(\beta^2)$ shows that the conditions for this multiplication table to define a ring are
\[
  \ell = -ac, \quad m = -ad, \quad n = -bd.
\]
In particular, each choice of $a$, $b$, $c$, and $d$ yields precisely one ring structure, showing the preliminary claim.

Switching to a different basis $[\bar\alpha', \bar\beta'] = [p\alpha + q\beta, r\alpha + s\beta]$ multiplies the distinguished generator of $\Lambda^3 C$ by the determinant $ps - qr$ and thus changes the form $\phi$ in the manner indicated in \eqref{eq:twisted action}. This proves the bijection between cubic rings and $\GL_2 A$-orbits of cubic forms.

An $A$-algebra automorphism $\sigma$ of the ring $C$ clearly induces an automorphism of the module $C/A$ such that the cubic forms induced by bases $[\bar\alpha,\bar\beta]$ and $[\sigma(\bar\alpha),\sigma(\bar\beta)]$ are the same (for any fixed basis $[\bar\alpha, \bar\beta]$. Conversely, if some $\sigma : C/A \to C/A$ has this property, it arises from a unique automorphism of $R$, namely the linear map that sends the normal basis lifting $[\bar\alpha,\bar\beta]$ to the normal basis lifting $[\sigma(\bar\alpha),\sigma(\bar\beta)]$. This establishes a bijection between the automorphism groups, which is easily seen to be a group isomorphism.
\end{proof}

We will have occasion to use the Levi form $\phi$ in many contexts: sometimes as a coordinate-free map $\phi : C/A \to \Lambda^3 C$, sometimes in a specific basis as a polynomial $\phi : A^2 \to A$. Sometimes we will be plugging an element of $C/A$ into $\phi$, but treating the output as a number in $A$; this requires one to choose a generator $\omega_C$ of $\Lambda^3 C$, otherwise known as an \emph{orientation} on $C$, and we write
\[
  \phi(\xi) = \frac{1}{\omega_C} 1 \wedge \xi \wedge \xi^2.
\]

Happily enough, the Levi form corresponding to a monogenic ring $A[\xi]/(\xi^3 + b\xi^2 + c\xi + d)$ is simply the homogenized form
\[
  \phi(x,y) = x^3 + bx^2 y + cx y^2 + dy^3
\]
(take the normal basis $[1,\xi,\xi^2+b\xi+c]$). This leads to a quick proof of the identity that the discriminant of the ring $C$ corresponding to a form $\phi$ is the usual polynomial discriminant
\begin{equation}\label{eq:discriminant}
  \Disc \phi = b^2c^2 - 4ac^3 - 4b^3d - 27a^2d^2 + 18abcd,
\end{equation}
just by noting that both sides are homogeneous polynomials in $a$, $b$, $c$, and $d$ of degree $4$ that coincide when $a = 1$. Note that this immediately implies Stickelberger's theorem that (when $A = \ZZ$) $\Disc C \equiv 0,1$ mod $4$, since
\[
  \Disc \phi \equiv (bc - ad)^2 \equiv 0 \text{ or } 1 \mod 4.
\]

The Levi parametrization has one other beautiful property, mentioned by Davenport and Heilbronn (cf{.} \cite{DavII}, Lemma 11), who developed the Levi form in a different manner \cite{DavI}: if $C$ is a maximal cubic $\ZZ$-algebra, then for any prime $p \in A$, the splitting type of $C$ at $p$ is the same as the splitting type of $\phi$ modulo $p$. In other words, the prime ideals lying above $p$ in $C$ can be put in bijection with the distinct linear factors of $\phi$ in such a way that the inertia and ramification indices, on the one hand, equal the degrees and multiplicities on the other. This can be proved using the fact that all maximal cubic $\ZZ_p$-algebras are monogenic, except $\ZZ_2 \cross \ZZ_2 \cross \ZZ_2$ which is directly seen to correspond to $\phi(x,y) = xy(x+y)$.

\subsection{\zmat{} rings}

Just as a quadratic form can be represented by a symmetric matrix, a binary cubic form $\phi$ can be represented by a triply symmetric cubical box
\begin{equation}\label{eq:triply symmetric box}
  \bbq{a}{b/3}{b/3}{c/3}{b/3}{c/3}{c/3}{d}
\end{equation}
that has integer entries exactly when $3|b$ and $3|c$, in which case we call $\phi$ an \emph{integer-matrix form}, or a \zmat{} form for short. It is not hard to see that this property is $\GL_2\ZZ$-invariant. The following proposition shows the link with \zmat{} rings as we previously defined them.

\begin{prop}\label{prop:zmat}
Let $C$ be a cubic ring. The following are equivalent:
\begin{enumerate}[$(a)$]
  \item \label{it:zmat} The cubic form corresponding to $C$ is \zmat{};
  \item \label{it:3|trace} The trace of every element of $C$ is a multiple of $3$;
  \item \label{it:Z+C0} $C = \ZZ \oplus C^0$, where $C^0 \subseteq C$ is the subgroup of elements having trace zero.
\end{enumerate}
\end{prop}
\begin{proof}
The implications $\eqref{it:3|trace} \Leftrightarrow \eqref{it:Z+C0}$ are straightforward. For $\eqref{it:zmat} \Leftrightarrow \eqref{it:3|trace}$, write the multiplication table of $C$ in terms of a normal basis:
\begin{align*}
  \alpha^2 &= -ac - b\alpha + a\beta \\
  \alpha\beta &= -ad \\
  \beta^2 &= -bd - d\alpha + c\beta.
\end{align*}
The trace of $\alpha$ may of course be computed by adding the coefficients of $x$ in $\alpha x$ for $x$ in the basis $[1,\alpha,\beta]$. Since $\alpha\cdot 1$ has no constant term and $\alpha\beta$ has no $\beta$ term, we get $\tr \alpha = -b$, and likewise $\tr \beta = c$. So the traces of all elements of $C$ are multiples of $3$ if and only if $3|b$ and $3|c$, i.e{.} the corresponding form is \zmat{}.
\end{proof}

\subsection{The maximal \zmat{} subring}

It is well known that every nondegenerate cubic ring $C$ sits in a unique maximal cubic ring, namely the integral closure of $\ZZ$ in the corresponding $\QQ$-algebra $K = C \tensor_\ZZ \QQ$. The corresponding theorem for \zmat{} rings is also true.

\newcommand{\zm}{{\ensuremath{\ZZ}\text{m}}}
\begin{prop} \label{prop:zmat max}
Let $C$ be a cubic ring. The family of \zmat{} rings lying in $C$ has a unique maximal element $C^\zm$ in which all others are contained.
\end{prop}
\begin{proof}
A first guess would be to let $C^\zm$ be the set of elements of $C$ whose trace is divisible by $3$, but these do \emph{not} in general form a ring. Instead, let
\[
  C^\zm = \{x \in C | x^3 \in \ZZ + 3C\}.
\]
We verify the three desired properties:
\begin{enumerate}
  \item \emph{$C^\zm$ is a ring.} Clearly $C$ contains the integers and is closed under multiplication. If $x,y\in C$, then
  \[
    (x+y)^3 = x^3 + y^3 + 3(x^2y + xy^2) \in \ZZ + 3C
  \]
  so $x+y \in C^\zm$.
  \item \emph{$C^\zm$ is \zmat{}.} Given $x \in C$, pick $n \in \ZZ$ such that $x^3 \equiv n$ mod $3C$; then $(x - n)^3 \equiv 0$ mod $3C$. On the $\ZZ/3\ZZ$-module $C/3C$, the multiplier $x-n$ acts nilpotently and thus has trace zero. Thus $3|\tr(x-n)$, and thus $3|\tr x$.
  \item \emph{Any \zmat{} subring of $C$ lies in $C^\zm$.} If $x$ lies in a \zmat{} subring, then $3|\tr x$ and also $3|\tr x^2$. Thus the characteristic polynomial of $x$ modulo $3$ has the form $t^3 - n$, so $x^3 \equiv n$ mod $3C$ and hence $x \in C^\zm$. \qedhere
\end{enumerate}
\end{proof}
If $C$ is any nondegenerate \zmat{} ring, then there is a largest \zmat{} ring containing $C$, namely $C_0^{\zm}$, where $C_0$ is the maximal cubic ring containing $C$. We call $C_0^{\zm}$ a \emph{maximal \zmat{}} ring, to be distinguished from a \zmat{} maximal ring (that is, a maximal ring that is \zmat{}).

Although we have worked for convenience only over $\ZZ$, the foregoing theory of \zmat{} rings is applicable without change over $\ZZ_3$. (Of course, if $p \neq 3$, every cubic ring over $\ZZ_p$ is \zmat{}.)

\section{Reducing to the case that $D$ has no high prime powers}
\label{sec:recursions}
For the first section of our proof, we will tackle a step that occupies the last section of Nakagawa's treatment: eliminating all $D$ with high prime power factors by means of a recursion that expresses both $h(D)$ and $\hat{h}(D)$ in terms of simpler discriminants.

The main result of this section is as follows:
\begin{thm}\label{thm:recursion} \label{thm:subring recn} \label{thm:recn}
For all $D \in \Discs$ and all primes $p$,
\begin{align}
  h(p^6 D) &= h(p^4 D) + p \cdot \big( h(D) - h(D/p^2) \big) \label{eq:h recur} \\
  \hat h(p^6 D) &= \hat h(p^4 D) + p \cdot \big( \hat h(D) - \hat h(D/p^2) \big), \label{eq:hhat recur}
\end{align}
using the natural convention that $h(a) = \hat h(a) = 0$ for all $a \notin \Discs$.
\end{thm}
\begin{rem}\label{rem:recursion}If $D = D_0 f^2$ with $D_0$ fundamental, and if $p^3|f$ for some prime $p$, then this result lets us prove Theorem \ref{thm:main} in the case $\Delta = D$, given the cases $\Delta = D/p^2$, $D/p^6$, and $D/p^8$. Inducting on $|\Delta|$ (all cases with $\Delta \notin \ZZ$ being trivial) allows us to assume that $\Delta = \Delta_0 f^3$ with $\Delta_0$ fundamental and $f$ cubefree in Theorem \ref{thm:main}.
\end{rem}

\begin{proof}[Proof of Theorem \ref{thm:recursion}]
We prove more strongly that for each cubic algebra $C_0$ over $\ZZ$ that is maximal (resp{.} maximal \zmat{}) at $p$, the contributions to the left and right sides of \eqref{eq:h recur} (resp{.} \eqref{eq:hhat recur}) coming from subrings $C \subseteq C_1$ of $p$-power index are equal. Here is the first of many times that the $1/\lvert\Aut C\rvert$ weighting in Theorem \ref{thm:main} is to our advantage: since every automorphism of such a $C$ lifts to an automorphism of $C_1$, we have the identity
\[
   \frac{1}{\lvert\Aut C\rvert} = \frac{ \lvert\{C' \subseteq C_1 : C \cong C'\}\rvert}{\lvert\Aut C_1\rvert}
\]
and we can simply count subrings of $C_1$ without worrying whether they are isomorphic or have automorphisms. (If $C_1$ is \zmat{}, all its finite-index subrings will also be, by definition.)

The enumeration of subrings of a fixed ring is a local problem, and without further ado we will let $C_1$ denote a maximal (resp{.} maximal \zmat{}) nondegenerate cubic algebra over $\ZZ_p$ and $s_n$ the number of subrings of $C_1$ of index $p^n$. In particular $s_0 = 1$ and $\cdots = s_{-2} = s_{-1} = 0$. It suffices to prove the recursion
\begin{equation} \label{eq:subring recn}
  s_{n+3} = s_{n+2} + p\big(s_{n} - s_{n-1}\big)
\end{equation}
for all $n$ that are big enough for
\begin{equation}
  p^{2n} \Disc C_1, \quad \text{resp.} \quad {-\frac{1}{27}} \cdot p^{2n} \Disc C_1 
\end{equation}
to be a discriminant, that is, a $p$-adic integer congruent to $0$ or $1$ mod $4$ (the latter condition being vacuous unless $p = 2$). Clearly all $n \geq 0$ satisfy this condition; we will discover that $n=-1$ and $-2$ sometimes do, and $n \leq -3$ never does (thankfully, as \eqref{eq:subring recn} is clearly false for $n = -3$).

If $C \subseteq C_1$ is a subring of index $p^n$, then $C_1/C$ is a quotient group of $C/\ZZ_p \cong \ZZ_p \oplus \ZZ_p$ and thus has at most two elementary divisors. Write
\[
  C_1/C \cong \ZZ/p^i \ZZ \oplus \ZZ/p^j \ZZ
\]
where $0 \leq i \leq j$ are integers with $i + j = n$. Using this isomorphism, we get a normal basis $[1,\alpha,\beta]$ for $C_1$ such that $[1,p^i \alpha, p^j \beta]$ is a basis for $C$, manifestly also normal. One then computes that if
\[
  \phi_0(x,y) = a x^3 + b x^2 y + c x y^2 + d y^3
\]
is the cubic form attached to $C_1$ in the basis $[1,\alpha,\beta]$, then the corresponding cubic form attached to $C$ is
\begin{equation} \label{eq:subring phi}
  \phi_C(x,y) = a p^{2i - j} x^3 + b p^i x^2 y + c p^j x y^2 + d p^{2j - i}.
\end{equation}
In particular, if $2i \geq j$, then this form has integer coefficients and so $C$ will be a ring no matter what normal basis $[1, \alpha, \beta]$ we pick. Otherwise we must impose the condition that $a = \phi(\bar\alpha)$ is divisible by $p^{j-2i}$.

Of course, different normal bases $[1,\alpha, \beta]$, or equivalently, different bases $[\bar\alpha, \bar\beta]$ for the lattice $L_1 = C_1/\ZZ_p$, may yield the same ring $C$, which is determined by the lattice
\[
  L_C = p^{-i}(C/\ZZ_p) = \< \bar\alpha, p^{j-i}\bar\beta \> = p^{j-i}L_1 + \< \bar\alpha \>.
\]
In particular, the vector $\beta$ is immaterial, and $\alpha$ may range over all vectors of $L_1$ not divisible by $p$, up to translation by $p^{j-i}L_1$ and scaling by units. In other words, the parameter space for $\alpha$ is the finite projective line $\PP^1(\ZZ/p^{j-i}\ZZ)$, and $s_n$ is the total number of solutions to
\begin{equation}\label{eq:proj cong}
  \phi(x,y) \equiv 0 \mod p^{\max\{j-2i,0\}}
\end{equation}
for $[x : y] \in \PP^1(\ZZ/p^{j-i}\ZZ)$, where $(i,j)$ ranges over integer pairs with $0 \leq i \leq j$ and $i + j = n$.

The key point to note is that replacing $(i,j)$ with $(i+1,j+2)$ does not change the condition \eqref{eq:proj cong}, but gives us a projective line $\PP^1(\ZZ/p^{j-i}\ZZ)$ with $p$ points lying over every point that was there before. We get $s_n \approx ps_{n-3}$, subject to three corrective terms (compare Figure \ref{fig:ij}):
\begin{itemize}
  \item When $j = i+1$, $\PP^1(\ZZ/p\ZZ)$ has $p+1$ points instead of $p$, contributing an extra point for $n\geq 3$ odd;
  \item When $i = j$, the pair $(i,i)$ is inaccessible by this translation and contributes $1$ $(= |\PP^1(\ZZ/1\ZZ)|$) extra point for $n$ even;
  \item The pair $(i,j) = (0,n)$ is also inaccessible by this translation and contributes $r_n$ points, where $r_n$ is the number of solutions to $\phi(x,y) \equiv 0$ mod $p^n$ in $\PP^1(\ZZ/p^n\ZZ)$.
\end{itemize}
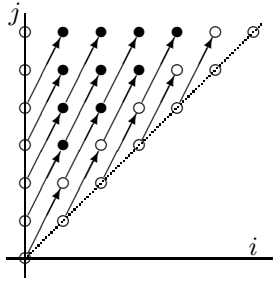
\begin{figure}
\[
\setlength{\unitlength}{0.5cm}
\begin{picture}(7,7.2)(-0.5,-0.5)
\put(-0.5,0){\line(1,0){7}}
\put(0,-0.5){\line(0,1){7}}
\dottedline{0.1}(0,0)(6.2,6.2)
\multiput(0,0)(0,1){7}{\circle{0.3}}
\multiput(1,1)(1,1){6}{\circle{0.3}}
\multiput(1,2)(1,1){5}{\circle{0.3}}
\multiput(1,3)(1,1){4}{\circle*{0.3}}
\multiput(1,4)(1,1){3}{\circle*{0.3}}
\multiput(1,5)(1,1){2}{\circle*{0.3}}
\multiput(1,6)(1,1){1}{\circle*{0.3}}
\multiput(0.075,0.15)(0,1){5}{\vector(1,2){0.85}}
\multiput(1.075,1.15)(0,1){4}{\vector(1,2){0.85}}
\multiput(2.075,2.15)(0,1){3}{\vector(1,2){0.85}}
\multiput(3.075,3.15)(0,1){2}{\vector(1,2){0.85}}
\multiput(4.075,4.15)(0,1){1}{\vector(1,2){0.85}}
\put(6,0.1){\makebox(0,0)[b]{$i$}}
\put(-0.1,6.5){\makebox(0,0)[r]{$j$}}
\end{picture}
\]
\caption{Valid $(i,j)$ pairs. Solid dots indicate where the contribution to the number of subrings can be computed by multiplying by $p$ the number coming from $(i-1,j-2)$.}
\label{fig:ij}
\end{figure}
Thus, for $n \geq 2$,
\[
  s_n = ps_{n-3} + 1 + r_n,
\]
and in particular, for $n \geq 0$,
\[
  s_{n+3} = s_{n+2} - p\big(s_{n} - s_{n-1}\big) + r_{n+3} - r_{n+2}.
\]
Thus proving the desired recursion \eqref{eq:subring recn} for $n \geq 0$ is equivalent to showing that $r_m$ is constant for $m \geq 2$. For large $m$ this follows from a suitably strong version of Hensel's lemma; in our situation, some remarkable circumstances converge to give the results for the $n$ that we desire.

We also have $s_2 = 1 + r_2$ and $s_1 = r_1$ by a direct determination of the $(i,j)$ pairs involved. Hence \begin{align}
  \eqref{eq:subring recn} \text{ holds for } n=-1 &\iff r_2 = r_1 - 1 \label{eq:n=1} \\
  \eqref{eq:subring recn} \text{ holds for } n=-2 &\iff r_1 = 1. \label{eq:n=0}
\end{align}

Suppose first that $C_1$ is maximal. Let
\[
  \phi_0(x,y) = ax^3 + bx^2y + cxy^2 + dy^3
\]
be its associated cubic form. Suppose that we are given a root of $\phi_0$ in $\PP^1(\ZZ/p^2\ZZ)$; choose our basis $[1,\alpha,\beta]$ of $C_1$ so that it is at $[1:0]$, so $p^2 | a$. If $p | b$, then applying the formula \eqref{eq:subring phi} with $i=-1$ and $j=0$ shows that $\<1, p^{-1}\alpha, \beta\>$ is a ring, contradicting the maximality of $C$. So $[1:0]$ is a simple root and thus has a unique lift mod all $p^m$ by Hensel's Lemma, proving \eqref{eq:subring recn} for $n \geq 0$.

The cases for $n < 0$ only pop up when $p|\Disc C_1$, that is, $C_1$ is ramified. This can happen either when $C_1 \cong \ZZ_p \cross \OO_{K_2}$, where $K_2$ is a ramified quadratic extension of $\ZZ_p$, or $C_1 = \OO_{K_3}$ where $K_3$ is a totally ramified cubic extension of $\ZZ_p$. But in the former case, $\Disc C_1 = \Disc K_2$ is fundamental, so we still only have to prove $n \geq 0$.

In the totally ramified cubic case, we have $\Disc C_1 \leq 5$ by the Dedekind-Hensel bound (which in general says that for $L/K$ an extension of local fields, $v_K(\Disc_K L) \leq e - 1 + e v_K(e)$ where $e$ is the ramification index). So $n \geq -2$. Mod $p$, $\phi_0$ has a single root of multiplicity $3$ (because the splitting type of $C$ is $1^3$); mod $p^2$, $\phi_0$ has no roots, or else $C_1$ would be non-maximal as was just shown. So \eqref{eq:n=1} and \eqref{eq:n=0} both hold, which shows \eqref{eq:subring recn} for $n = -1$ and $-2$.

This completes the proof of \eqref{eq:subring recn} for $C_1$ maximal, and thus also the proof of \eqref{eq:h recur}. There remains the case that $p=3$ and $C_1$ is the maximal \zmat{} subring in a maximal ring $C_0$ that is not \zmat{}. Note that we are now proving \eqref{eq:hhat recur}, so $n$ is governed by the stronger inequality
\[
  -\frac{1}{27} \cdot 3^{2n} \Disc C_1 \in \ZZ_3,
\]
that is,
\[
  2n - 3 + v_3(\Disc C_1) \geq 0.
\]
Note also that $[C_0 : C_1]$ is either $3$ or $9$ since $\ZZ + 3C_0 \subseteq C_1$.

Consider first the case that $[C_0 : C_1] = 9$, that is, $C_1 = \ZZ_3 + 3C_0$. Note that $C_0$ must be unramified since otherwise there is an element $\xi \notin \ZZ + 3C_0$ whose cube lies in $3C_0$, contradicting the construction of the maximal \zmat{} subring. So $p \nmid \Disc C_0$, yielding $v_3(\Disc C_1) = 4$ and $n \geq 0$. Now the form $\phi_0$ corresponding to $C_1$ is $3$ times the form $\phi_1$ corresponding to $C_0$ (by \eqref{eq:subring phi} with $i = j = -1$) and so $r_m$ is simply $3$ times the number of roots of \[
  \phi_1(x,y) \equiv 0 \mod 3^{m-1}
\]
on $\PP^1(\ZZ/3^{m-1}\ZZ)$, which is constant for $m \geq 2$ by Hensel's lemma.

In the case $[C_0 : C_1] = 3$, the relationship between the corresponding forms $\phi_0$ and $\phi_1$ is governed by \eqref{eq:subring phi} with $i = -1$ and $j = 0$, so we can write
\[
  \phi_0(x,y) = 9a'x^3 + 3b'x^2y + 3c'xy^2 + dy^3 \textand \phi_1(x,y) = a'x^3 + b'x^2y + 3c'xy^2 + 3dy^3
\]
where $a',b',c',d \in \ZZ_3$. Note that $3\nmid d$ since $C_0$ is maximal. So the only root of $\phi_0$ mod $3$ is $[1:0]$, and the roots mod $3^m$ must be expressible in the form $[1:3y']$, $y' \in \ZZ/3^{m-1}\ZZ$. Note that
\[
  \phi_0(1,3y') \equiv 0 \mod 3^m  \iff  \phi_1(1,y') \equiv 3^{m-1}.
\]
Now $[0:1]$ is a root of $\phi_1$ mod $3$, with multiplicity exactly $2$: we have $3\nmid b'$ since $C_0$ is not \zmat{}. So there is a single simple root of the form $[1:y]$ modulo $3$. By Hensel's lemma, there is a single root of this form modulo all higher powers of $3$, yielding $3$ roots of $\phi_0$ modulo $3^m$ for all $m \geq 2$. This proves \eqref{eq:subring recn} for all $n \geq 0$, which is all that is needed: for we have shown that $C_0$ has splitting type $1^2 1$ and so $v_3(\disc C_1) = 3$.
\end{proof}

\begin{rem}\label{rem:subrings}
This proof also shows that, if $C_1$ is maximal, the initial terms $s_1$, $s_2$ of the recursion can be computed using only the splitting type $\sigma$ of $\phi_0$ at $p$: $s_1$ is the number of roots mod $p$, and $s_2$ is $1$ plus the number of \emph{simple} roots mod $p$, as these are the only ones that lift to mod $p^2$. The values of these numbers are tabulated below for future reference.
\begin{equation} \label{eq:subring table}
  \begin{tabular}{c|ccccc}
    $\sigma$ & $111$ & $12$ & $3$ & $1^2 1$ & $1^3$ \\ \hline
    $s_1$ & $3$ & $1$ & $0$ & $2$ & $1$ \\ \hline
    $s_2$ & $4$ & $2$ & $1$ & $2$ & $1$
  \end{tabular}
\end{equation}
Together with $s_0 = 1$ and $s_{-1} = 0$, they enable the computation of the number of subrings of any index of a maximal cubic ring over $\ZZ_p$ (or, indeed, over $\ZZ$).
\end{rem}

Incidentally, the recursion \eqref{eq:subring recn} can be solved explicitly to get a formula
\[
  s_n = \frac{
    p^{\lfloor \frac{n+3}{3} \rfloor} - 1
    + \left(s_1 - 1\right)\left(p^{\lfloor \frac{n+2}{3} \rfloor} - 1\right)
    + \left(s_2 - s_1\right)\left(p^{\lfloor \frac{n+1}{3} \rfloor} - 1\right)
  }{p-1}
\]
(cf{.} \cite{Nakagawa}, Lemma 3.7), but this will be less useful to us.

\section{$\hat{h}$ and self-balanced ideals}
\label{sec:Bhargava}
Many readers will no doubt have seen Bhargava's dazzling reinterpretation of Gauss's $200$-year-old composition law on binary quadratic forms \cite{B1}: a cube
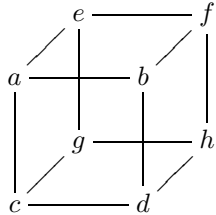
\[
  \bbq abcdefgh
\]
corresponds to a triple of quadratic forms whose Gauss composite is $0$, or more generally to three fractional ideals of a quadratic order that are ``balanced,'' meaning that their product is nearly the unit ideal in a suitably defined sense. Here, our focus is on the triply symmetric cubes
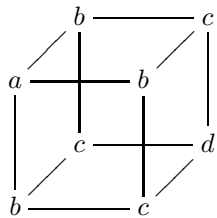
\[
  \bbq abbcbccd
\]
which we have already mentioned as the natural pictorial avatars of \zmat{} cubic forms. Due to the symmetry, these cubes correspond in Bhargava's bijection to ``balanced'' triples consisting of three ideals in the \emph{same} class; only this latter bijection need be described in detail here.

\begin{defn}
A \emph{self-balanced triple} is a triple $(\OO, I, \gamma)$, where $\OO$ is an order in a quadratic $\QQ$-algebra $K_2$, $I$ is a fractional ideal of $\OO$, and $\gamma \in K_2^\cross$ is a scalar, satisfying the two conditions
\begin{gather}
  \gamma I^3 \subseteq \OO \label{eq:balanced contain} \\
  \quad \lvert N(\gamma) \rvert \cdot N(I)^3 = 1. \label{eq:balanced norm}
\end{gather}
Also define an equivalence relation on self-balanced triples by
\[
  (\OO, I, \gamma) \sim (\OO, \lambda I, \lambda^{-3}\gamma)
\]
for every $\lambda \in K_2^\cross$. (It is immediate that the second triple is self-balanced if the first is.)
\end{defn}

Recall that an \emph{oriented} cubic ring $C$ is one with a distinguished generator $\omega_C \in \Lambda^3 C$, enabling us to view its Levi form $\phi : C/\ZZ \to \Lambda^3 C$ as taking values in $\ZZ$. A cubic ring $C$ can be oriented in two ways, which are isomorphic if and only if $C$ has an orientation-reversing automorphism; thus there are precisely $2\hat h(\Delta)$ \zmat{} cubic rings of discriminant $-27\Delta$, if we weight by the reciprocal of the number of \emph{oriented} automorphisms.

We are now ready to state the pertinent bijection.
\begin{thm}[cf{.} \cite{B1}, Theorem 13]\label{thm:Bha}
Oriented \zmat{} cubic rings $C$ of discriminant $-27\Delta \neq 0$ are in bijection with equivalence classes of self-balanced triples of the quadratic order $\OO_\Delta$ of discriminant $\Delta$. Also, those $C$ having a nontrivial oriented automorphism, necessarily of order $3$, correspond to those equivalence classes having a representative
\[
  (\OO_\Delta, \ZZ[\omega], \gamma),
\]
where $\ZZ[\omega]$ is the unit ideal in the ring generated by a primitive $3$rd root of unity (clearly $\Delta$ must be $-3$ times a square for this to happen).
\end{thm}
\begin{proof}
For a hands-on proof (that also works when $\Disc C = 0$), see \cite{B1}. Here we present a new proof based on that most ancient nexus between quadratic and cubic number fields: the Tartaglia-Cardano cubic formula.

Let $C$ be a nondegenerate oriented \zmat{} cubic ring. By Proposition \ref{prop:zmat}\eqref{it:Z+C0}, $C = \ZZ \oplus C^0$, where $C^0$ is the sublattice of elements of trace $0$. Pick a generic element $\alpha \in C^0$; specifically, we should have that
\begin{itemize}
  \item $1 \wedge \alpha \wedge \alpha^2 \neq 0$, that is, $[1, \alpha, \alpha^2]$ is a $\QQ$-basis of $C \tensor \QQ$; and
  \item $\tr(\alpha^2) \neq 0$, for reasons that will soon be clear.
\end{itemize}
Using the nondegeneracy of $C$, these conditions are not hard to fulfill. They are also homogeneous, and there is no harm in taking $\alpha$ a \emph{primitive} element, that is, one such that $\QQ\alpha \intsec C = \ZZ\alpha$.

Then $\alpha$ has characteristic polynomial $\alpha^3 + 3t \alpha + u = 0$, where the $\alpha^2$ term vanishes because $\tr \alpha = 0$, and $t$ is an integer because $\tr \alpha^2 = 6t'$ must be a multiple of $3$. We can now ``solve'' for $\alpha$ using the Tartaglia-Cardano formula:
\newcommand{\ga}{\sqrt[3]{\gamma}}
\newcommand{\gb}{\sqrt[3]{\bar{\gamma}}}
\begin{equation}\label{eq:Cardano}
  \alpha = \ga + \gb,
\end{equation}
where
\[
\gamma = \frac{-u + \sqrt{u^2 + 4t^3}}{2}
  \quad \text{and} \quad \bar\gamma = \frac{-u - \sqrt{u^2 + 4t^3}}{2}.
\]
If $C$ admits an embedding into $\CC$, this is literally true, provided that we choose the cube roots such that their product is $t$. In general, we can interpret the expression as follows. First note that the polynomial $x^3 + 3tx + u$ has discriminant $-27(u^2 + 4t^3)$, whence
\[
  -27(u^2 + 4t^3) = \Disc \ZZ[\alpha] = [C : \ZZ[\alpha]]^2 \cdot \Disc C = \phi(\alpha)^2 \cdot (-27\Delta),
\]
where $\phi(\xi) = \frac{1}{\omega_C} 1 \wedge \xi \wedge \xi^2$ is the Levi form of $C$. Thus we can view $\sqrt{u^2 + 4t^3} = \phi(\alpha) \sqrt{\Delta}$, and hence $\gamma$ and $\bar{\gamma}$, as elements of the nondegenerate quadratic algebra $K_2 = \QQ[\sqrt{\Delta}]$ canonically associated to $C$. Then in the sextic algebra $K_6 = K_2[\ga]$, $\ga$ is invertible (because $\gamma\bar\gamma = t^3$ is invertible) and the element $\gb = t/\ga$ is a cube root of $\bar{\gamma}$. Then, by the usual derivation of the cubic formula, $\alpha \mapsto \ga + \gb$ identifies $C$ with a cubic subring of $K_6$.

We have
\[
  \left(\ga\right)^2 = \frac{1}{t} \left(\ga\right)^3 \gb
  = \frac{1}{t} \gamma \gb,
\]
so
\[
  \alpha^2 = \left(\ga\right)^2 + 2 \ga\gb + \left(\gb\right)^2 = 2t + \frac{1}{t}(\bar{\gamma} \ga + \gamma \gb),
\]
and since $[1,\alpha,\alpha^2]$ is a $\QQ$-basis of $C \tensor_\ZZ \QQ$, we see that
\[
  C \tensor_\ZZ \QQ \cong \QQ \oplus \{\xi \ga + \bar{\xi} \gb \mid \xi \in K_2\}
\]
and hence
\begin{equation}\label{eq:Card}
  C = \ZZ \oplus C^0 \cong \ZZ \oplus \{\xi \ga + \bar{\xi} \gb \mid \xi \in I\}
\end{equation}
for some lattice $I \subset K$. For brevity we write $c(\xi) = \xi \ga + \bar{\xi} \gb$, so $c : K_2 \to C^0 \tensor \QQ$ is an isomorphism of $\QQ$-vector spaces.

Note that $1$ is a primitive vector in $I$, so $I$ has a basis $[1,\tau]$ where
\[
  \tau = \frac{s+\sqrt{\Delta}}{q}
\]
for some $s,q \in \QQ$, and $C$ has a basis $[1, c(1), c(\tau)] = [1, \alpha, c(\tau)]$. Let us choose the sign of $\tau$ such that the distinguished generator $1 \wedge \alpha \wedge c(\tau)$ of $\Lambda^3 C$ is the given $\omega_C$. Then
\begin{align*}
  \phi(\alpha) &= \frac{1}{\omega_C} 1 \wedge \alpha \wedge \alpha^2 \\
  &= \frac{1}{\omega_C} 1 \wedge \alpha \wedge \bigg(2t + c\Big(\frac{\bar\gamma}{t}\Big)\bigg) \\
  &= \frac{1}{\omega_C} 1 \wedge \alpha \wedge c\bigg(\frac{-\sqrt{u^2 + 4t^3}}{2t}\bigg) \\
  &= \frac{1 \wedge \alpha \wedge \frac{-\phi(\alpha) c(\sqrt{\Delta})}{2t}}{1 \wedge \alpha \wedge \frac{c(\sqrt{\Delta})}{q}} \\
  &= \frac{-q \phi(\alpha)}{2t},
\end{align*}
that is, $q = -2t$.

The multiplication law on $C$ is given by
\begin{equation}\label{eq:mult law}
  c(\xi)c(\eta) = t(\xi\bar{\eta} + \bar{\xi}\eta) + c\left(\frac{\bar{\gamma}\bar{\xi}\bar{\eta}}{t}\right);
\end{equation}
hence the conditions for $C$ to be a ring are that
\begin{gather}
t(\xi\bar{\eta} + \bar{\xi}\eta) \in \ZZ \label{eq:Z cond}
\intertext{and}
\frac{1}{t}\bar{\xi}\bar{\eta}\bar{\gamma} \in I \label{eq:I cond}
\end{gather}
for all $\xi,\eta \in I$. Plugging $\xi = 1$, $\eta = \tau$ in \eqref{eq:Z cond} yields $s \in \ZZ$; plugging $\xi = \eta = \tau$ in \eqref{eq:Z cond} yields
\[
  r := 2t\left(\frac{s + \sqrt{\Delta}}{-2t}\right)\left(\frac{s - \sqrt{\Delta}}{-2t}\right)
  = \frac{s^2 - \Delta}{2t} \in \ZZ.
\]
Consequently $s \equiv \Delta$ mod $2$, and the multiplier $(s - \Delta)/2$, which generates the order $\OO_\Delta$, takes $\tau$ to an integer $-r$. This shows that $I$ is an ideal of $\OO_\Delta$.

Condition \eqref{eq:balanced norm} is immediate, as $I$ has norm $1/|t|$.

We must now prove \eqref{eq:balanced contain}, namely that $\gamma I^3 \subseteq \OO_\Delta$. Since $\gamma I^2 \subseteq t \bar{I}$ by \eqref{eq:I cond}, it suffices to prove that $t I \bar{I} \subseteq \OO_\Delta$. But using the known $\ZZ$-basis,
\begin{align*}
  t I \bar{I} &= \< t, t \tau, t\bar{\tau}, t \tau \bar{\tau} \> \\
  &= \< t, \frac{s + \sqrt{\Delta}}{-2}, \frac{s - \sqrt{\Delta}}{-2}, \frac{s^2 - \Delta}{2t} \> \\
  &= \< t, s, r, \frac{s + \sqrt{\Delta}}{2} \>
\end{align*}
which clearly lies in $\OO_\Delta$. This completes the construction of a self-balanced triple $(\OO_\Delta, I, \gamma)$ corresponding to $C$.

Conversely, given a self-balanced triple $(\OO_\Delta, I, \gamma)$, we scale $I$ so that it contains $1$ as a primitive element (and scale $\gamma$ appropriately). Let $t \in \ZZ$ be determined by $N(I) = 1/|t|$ and $\sgn t = \sgn N(\gamma)$. Then $I$ has a basis
\[
  \left[1, \tau = \frac{s + \sqrt{\Delta}}{-2t}\right]
\]
for some $s \in \ZZ$ of the same parity as $\Delta$. We get from \eqref{eq:Card} a cubic ring $C$ with a distinguished element $\alpha = c(1)$ for which the foregoing process returns the given triple $(\OO_\Delta, I, \gamma)$, if we can prove that \eqref{eq:Z cond} and \eqref{eq:I cond} hold. The verification of \eqref{eq:Z cond} is a routine check on basis elements. For \eqref{eq:I cond}, it is convenient to use the identity
\[
  \frac{\xi \wedge \eta}{1 \wedge \tau} = \frac{\bar \xi \eta - \xi \bar \eta}{\tau - \bar \tau},
\]
which may be proved merely by noting that $\bar \xi \eta - \xi \bar \eta$ is a $\QQ$-linear, $\QQ \cdot \sqrt{\Delta}$-valued alternating $2$-form on $K_2$. Note that an element $\xi \in K_2$ belongs to $I$ if and only if $\xi \wedge \eta \in \Lambda^2 I$ for every $\eta \in I$. Now for every $\xi, \eta, \zeta \in I$,
\[
  \frac{\frac{\bar\gamma\bar{\xi}\bar{\eta}}{t} \wedge \zeta}{1 \wedge \tau}
  = \frac{\gamma\xi \eta \zeta - \bar{\gamma \xi \eta \zeta}}{t(\tau - \bar{\tau})}
  = \frac{\gamma\xi \eta \zeta - \bar{\gamma \xi \eta \zeta}}{-\sqrt{\Delta}} \in \ZZ
\]
since $\gamma \xi \eta \zeta \in \OO_\Delta$, proving \eqref{eq:Z cond}.

To show that the $C$ corresponding to a self-balanced triple $(\OO_\Delta, I, \gamma)$ is unique, it suffices to express the Levi form of $C$ in terms of the triple, which is not difficult:
\begin{equation} \label{eq:phi from I}
\begin{aligned}
  \phi(c(\xi)) &= \frac{1}{\omega_C} 1 \wedge c(\xi) \wedge c(\xi^2) \\
  &= \frac{1}{\omega_C}1 \wedge c(\xi) \wedge c\left(\frac{\bar{\xi}^2\gamma}{t}\right) \\
  &= \frac{\xi \wedge \frac{\bar{\xi}^2\gamma}{t}}{1 \wedge \tau} \\
  &= \frac{\xi^3\gamma - \bar \xi^3\gamma}{\sqrt{\Delta}}.
\end{aligned}
\end{equation}

It remains to show that the choice of $\alpha$ made at the outset does not change the self-balanced triple derived, up to equivalence. Suppose $(\OO_\Delta, I, \gamma)$ and $(\OO_\Delta, I', \gamma')$ both arose from this method, which also provides identifications of oriented $\ZZ$-modules $c: I \to C^0$, $c': I'\to C^0$ and, in particular, an isomorphism $\psi = c'^{-1} \circ c : I \to I'$.

Here we use a trick inspired by the trace forms of \cite{GMS}: plugging $\eta = \xi$ into \eqref{eq:mult law}, we see that
\[
  \frac{1}{6} \tr(c(\xi)^2) = t\xi\bar{\xi} = t N(\xi).
\]
Thus $N(\psi(\xi))/N(\xi)$ is a constant $t'/t$ for all $\xi$ (where $t'$ is the 
value of $t$ corresponding to placing $\alpha'$ in place of $\alpha$). In particular, $\psi(1)$ is invertible in $K_2$, and the normalized map
\[
  \tilde{\psi}(\xi) = \frac{\psi(\xi)}{\psi(1)}
\]
extends linearly to a $\QQ$-linear self-map of $K_2$ that preserves $1$ and norms. There are only two such, the identity and conjugation, and the latter is ruled out by the fact that $\psi$ respects orientation. So $\psi$ is a scaling $\xi \mapsto \lambda \xi$ for all $\xi$, and using \eqref{eq:phi from I}, it is easy to see that $\gamma' = \lambda^{-3} \gamma$ so the two self-balanced triples are equivalent.

By the same argument, a nontrivial oriented automorphism of $C$ arises if and only if the associated balanced triple $(\OO, I, \gamma)$ is equivalent to itself via scaling by some multiplier $\lambda \neq 1$. To leave $\gamma$ fixed, we must have $\lambda^3 = 1$ and so $I$ is an ideal of the order $\ZZ[\omega]$. Since $\ZZ[\omega]$ is a PID, this implies that $I = \ZZ[\omega]$ up to scaling, as stated. Conversely, if $I = \ZZ[\omega]$, the map $c(\xi) \mapsto c(\omega \xi)$ clearly defines a nontrivial automorphism of $C$. \end{proof}

Here ends our proof of Bhargava's Theorem 13, but for our purposes, a slightly transformed description of the parametrization is preferable. The ideal $I$ may or may not be invertible in $\OO_\Delta$. Indeed, with respect to a basis
\[
  I = \< 1, \dfrac{s + \sqrt{\Delta}}{2t} \>,
\]
we found that
\[
  t I \bar{I} = \< t, s, r, \frac{s + \sqrt{\Delta}}{2} \>,
\]
where $r = (s^2 - \Delta)/(2t)$ is an integer. If $t$, $s$, and $r$ are relatively prime (incidentally, they are the coefficients of the \emph{quadratic form} $tx^2 + sxy + ry^2$ associated to the class of $I$), then $t I \bar{I} = \OO_\Delta$ and so $I$ is invertible. However, in general, there may be a common factor $g = \gcd(t,s,r)$, and then one verifies that $I$ is an ideal of the order $\OO_{\Delta'}$, $\Delta' = \Delta/g^2$, with inverse $\bar{I}/(gt)$. Note that $\gamma I^3$ is an $\OO_{\Delta'}$-ideal contained in $\OO_\Delta$. We need a little lemma about such ideals:
\begin{lem}
Let $\Delta \in \Discs$ and $g \geq 1$. An ideal $I$ of $\OO_\Delta$ that is contained in $\OO_{\Delta g^2}$ is actually contained in $g\OO_\Delta$.
\end{lem}
\begin{proof}
Let $\OO_\Delta = \ZZ[\xi]$, so $\OO_{\Delta g^2} = \ZZ[g \xi]$. Suppose $\eta = a + b g \xi \in \OO_{\Delta g^2}$ is an element of $I$. Then multiplying by the conjugate $\bar \xi = \tr \xi - \xi \in \OO_\Delta$, we get that
\[
  \bar{\xi}\eta = a \bar{\xi} + b g N(\xi) = [b g N(\xi) + a \tr \xi] - a \xi
\]
belongs to $I$, and hence to $\OO_{\Delta g^2}$. So $g|a$, and thus $\eta \in g\OO_\Delta$.
\end{proof}

 Thus
\[
  J = \frac{\gamma I^3}{g}
\]
is an invertible integral ideal of $\OO_{\Delta'}$ of norm
\begin{align*}
  N(J) &= \frac{\lvert N(\gamma) \rvert N_{\OO_{\Delta'}}(I)^3}{g^2} \\
  &= \frac{\lvert N(\gamma) \rvert \left(g \cdot N_{\OO_\Delta}(I)\right)^3}{g^2} \\
  &= g.
\end{align*}
Conversely, if $J$ is an invertible ideal of $\OO_{\Delta'}$ of norm $g$ whose class in $\Pic \OO_\Delta'$ is a cube (a clearly necessary condition), then $J$ will in general correspond to a number of self-balanced triples $(\OO_\Delta, I, \gamma)$, where $\Delta = \Delta' g^2$. There are $\lvert \Pic(\OO_\Delta)[3] \rvert$ possibilities for the class of $I$, and for each $I$, the value of $\gamma$ is determined only up to units, whereas we have $(\OO_\Delta, I, \gamma) \sim (\OO_\Delta, I, \gamma')$ only when $\gamma/\gamma'$ is the \emph{cube} of a unit, yielding a further $\lvert \OO_{\Delta'}^\cross / (\OO_{\Delta'}^\cross)^3 \rvert$ possibilities. An appeal to the structure of the unit groups of quadratic fields shows that
\[
  \lvert \OO_{\Delta'}^\cross / (\OO_{\Delta'}^\cross)^3 \rvert =
  \begin{cases}
    3 & \text{if $\Delta' = -3$ or $\Delta'$ is a positive non-square} \\
    1 & \text{otherwise.}
  \end{cases}
\]
The exception at $\Delta' = -3$ is welcome, since these are precisely the cases where we must count the corresponding rings with weight $1/3$ owing to the nontrivial automorphism. We also get exceptional behavior for $\Delta'$ a positive non-square, in other words, for $\Delta$ a positive non-square. We summarize our findings as follows.

\begin{thm}[\cite{Nakagawa}, Theorem 2.6 is the case $\Delta < 0$] \label{thm:RHS}
Let $w_\Delta = 3$ if $\Delta$ is a square, $1$ otherwise. Also let $\eta_\Delta = 1/3$ if $\Delta$ is positive, $1$ if $\Delta$ is negative. The following quantities are equal:
\begin{itemize}
  \item $2w_\Delta \eta_\Delta \hat h(\Delta)$;
  \item The number of invertible ideals $J$ of norm $g$ whose class is a cube in orders $\OO_{\Delta'}$ for integers $g>0$, $\Delta'$ satisfying $\Delta'g^2 = \Delta$, each counted with weight
  \[
    \lvert\Pic(\OO_{\Delta'})[3]\rvert.
  \]
\end{itemize}
\end{thm}

In \cite{Nakagawa}, a more computational approach is used that centers on the fact that the quadratic form $tx^2 + sxy + ry^2$ attached to $I$ is actually the \emph{Hessian} of the cubic form attached to $C$, that is, the determinant of second partial derivatives, up to scaling.

\section{Interlude: Links with class field theory}
We pause for a moment to consider how Theorem \ref{thm:main} transforms using the elementary tools developed so far, and how in certain special cases one is led to the founding concerns of class field theory. We already have Theorem \ref{thm:RHS}, which relates $\hat h(\Delta)$ to ideals in quadratic orders. Although it will not be used in the sequel, a comparable description of $h(\Delta)$ is not so hard to come by. For simplicity we treat only the case $3 \nmid \Delta$.
\begin{prop}
Let $3 \nmid \Delta$. To compute $6 w_{-3\Delta} \eta_{-3\Delta} h(\Delta)$, add the contributions to $\hat h(-27\Delta)$ in Theorem \ref{thm:RHS} for which $3 \nmid g$.
\end{prop}
\begin{proof}
We can make any cubic form \zmat{} by multiplying it by $3$, that is, passing from the associated cubic ring $C$ to the subring $\ZZ + 3C$. We now want to count \zmat{} cubic forms of discriminant $81\Delta$ satisfying the additional condition $3|a$, $3|d$. Following this condition through the bijection of Theorem \ref{thm:Bha} shows that $6 w_{-3\Delta} h(\Delta)$ is the number of inequivalent balanced triples $(\OO_{-3\Delta}, I, \gamma)$ such that
\begin{equation} \label{eq:ga3}
  \gamma \alpha^3 \in \OO_{-27\Delta}
\end{equation}
for each $\alpha \in I$.

Suppose $(\OO_{-27\Delta}, I', \gamma)$ is a balanced triple such that $I'$ is \emph{not} an ideal of $\OO_{-3\Delta}$, that is, the corresponding $J$ in Theorem \ref{thm:RHS} has $3 \nmid g$. Then $I = I'\OO_{-3\Delta}$ is an ideal of index $3$ over $I$. The triple $(\OO_{-3\Delta}, I, \gamma)$ is clearly balanced, and any element $\alpha \in I$ can be written as $\kappa + \lambda\xi$, where $\kappa, \lambda \in I'$ and $\xi = \frac{-3\Delta + \sqrt{-3\Delta}}{2}$ is a generator of $\OO_{-3\Delta}$; one checks that $\xi^3 \in \OO_{-27\Delta}$, and thus
\[
  \alpha^3 = \kappa^3 + 3(\kappa^2 \lambda \xi + \kappa \lambda^2 \xi^2) + \xi^3 \lambda^3 \in \OO_{-27\Delta},
\]
verifying \eqref{eq:ga3}. Conversely, if $(\OO_{-3\Delta}, I, \gamma)$ is balanced and satisfies \eqref{eq:ga3}, then $I$ has four sublattices $I'$ of index $3$, one of which is $\pp I$ (using that $3 = \pp^2$ ramifies in $\OO_{-3\Delta}$). The other three are ideals of $\OO_{-27\Delta}$ but not of $\OO_{-3\Delta}$. Thus they yield triples $(\OO_{-27\Delta}, I', \gamma)$ which are balanced since we can write $I' = 3I + \ZZ\alpha_0$ and get
\[
  \gamma I'^3 \subseteq \gamma\ZZ\alpha_0^3 + 3\gamma I I'^2 \subseteq \OO_{-27\Delta} + 3\OO_{-3\Delta} = \OO_{-27\Delta}.
\]
So we have a $3$-to-$1$ correspondence between the balanced triples involved, establishing the desired identity.
\end{proof}
We now present two examples showing the sorts of problems we encounter when tackling Theorem \ref{thm:main} with both sides interpreted in this way.
\begin{examp}
If $\Delta = \Delta_0$ is a fundamental discriminant, then only the terms with $g = 1$ count on either side, and Theorem \ref{thm:main} devolves into
\[
  \frac{\lvert \Pic(\OO_{-27\Delta})[3] \rvert}{\lvert \Pic(\OO_\Delta)[3] \rvert} = \begin{cases}
    3, & \Delta > 1 \\
    1, & \Delta \leq 1.
  \end{cases}
\]
Since there is a surjection $\Pic(\OO_{-27\Delta}) \to \Pic(\OO_{-3\Delta})$ whose kernel has size $1$ or $3$, we get a corollary concerning the class groups of quadratic number fields:
\[
  \frac{\lvert \Cl(\QQ(\sqrt{-3\Delta}))[3] \rvert}{\lvert \Cl(\QQ(\sqrt{\Delta}))[3] \rvert} = 
  \begin{cases}
    3 \text{ or } 1, & \Delta > 1 \\
    1 \text{ or } \frac{1}{3}, & \Delta < 0.
  \end{cases}
\]
This is the \emph{Scholz reflection principle,} proved by Scholz in 1932 as a stunning application of class field theory.
\end{examp}
\begin{examp}
For an example that does not require going into high quadratic number fields, take $\Delta = p^2q^2$, where $p$ and $q$ are primes with $p \equiv 1$ mod $3$, $q \equiv 2$ mod $3$. Then verifying Theorem \ref{thm:main} reduces to counting ideals of various norms in suborders of $\ZZ \cross \ZZ$ and $\ZZ[\omega]$ (where $\omega$ is a primitive cube root of unity) and checking the cubicality of their classes in the Picard group. We present the outcomes here.

For $2\hat h(p^2 q^2)$, we count:
\begin{itemize}
  \item ideals of norm $pq$ in $\ZZ \cross \ZZ$: $4$, with weight $1$.
  \item ideals of norm $1$ in $\ZZ + pq(\ZZ \cross \ZZ)$: $1$, with weight $3$.
  \item ideals of norm $p$ in $\ZZ + q(\ZZ \cross \ZZ)$: $2$, with weight $1$.
  \item ideals of norm $q$ in $\ZZ + p(\ZZ \cross \ZZ)$: here things become interesting. The Picard group is $(\ZZ/p\ZZ)^\cross$; there are are $2$ such ideals (with weight $3$) if $q$ is a cube mod $p$, and none otherwise.
\end{itemize}
For $6 h(p^2 q^2)$, we count:
\begin{itemize}
  \item ideals of norm $pq$ in $\ZZ[3\omega]$: none.
  \item ideals of norm $1$ in $\ZZ[3pq\omega]$: $1$, with weight $27$.
  \item ideals of norm $q$ in $\ZZ[3p\omega]$: none.
  \item ideals of norm $p$ in $\ZZ[3q\omega]$: here again things become interesting. There are exactly $2$ ideals of $\OO = \ZZ[3q\omega]$ of norm $p$, namely the intersections with $\OO$ of the two ideals $\ZZ[\omega]\alpha$, $\ZZ[\omega]\bar\alpha$ into which $p$ splits in $\ZZ[\omega]$. They are cubes in the class group $\Pic(\ZZ[3q\omega] \cong \FF_{q^2}^\cross/\FF_q^\cross \cross \ZZ/3\ZZ$ if and only if $\alpha$ (equivalently $\bar\alpha$) or one of its associates $\omega\alpha$, $\omega\alpha^2$ is a cube modulo $3q$ (or an integer times a cube, but all integers mod $3q$ are cubes). The mod $3$ condition requires we pick the unique associate (up to sign) with $\alpha \in \ZZ[3\omega]$, that is, $\alpha$ is \emph{primary} in the classical terminology; and then we get a contribution of $12$ or $0$ according as this $\alpha$ is a cube or not modulo $q$.
\end{itemize}
So verifying Theorem \ref{thm:main} in this case amounts to proving a case of cubic reciprocity: that $\alpha$ is a cube mod $q$ if and only if $q$ is a cube mod $\alpha$. Similar analysis of the case $p \equiv q \equiv 1$ mod $3$ forces us to invoke cubic reciprocity on two generic elements $\alpha, \beta \in \ZZ[3\omega]$. Although elementary proofs of cubic reciprocity are known, we can then proceed to the case $\Delta = -p^2q^2$, which leads us to an exotic cubic reciprocity law linking the fields $\QQ(i)$ and $\QQ(\sqrt{3})$. The quest to systematize such reciprocity laws was, of course, one of the founding aims of class field theory.
\end{examp}

\section{$h$ and class field theory}
\label{sec:CFT}
We now return to the general case and seek to interpret $h(\Delta)$ via class field theory. Consider first the most generic case, in which our given cubic ring $C$ sits in a cubic field $K_3$ which is not Galois over $\QQ$ (so $\Delta$ is not a square). Then the normal closure $K_6$ of $K_3$ is $S_3$-Galois; it contains a single quadratic subfield $K_2 = \QQ(\sqrt{\Delta})$ of discriminant $\Delta_0$, the fundamental discriminant arising from decomposing $\Delta = \Delta_0 d^2$. The key insight regarding this network of fields
\begin{equation}\label{eq:4 fields}
\xymatrix@dr{
K_6 \ar@{-}[r]\ar@{-}[d] & K_2 \ar@{-}[d] \\
K_3 \ar@{-}[r]
 & \QQ }
\end{equation}
is the following theorem of Hasse. Recall that the \emph{conductor} of an abelian extension $L/K$ of number fields is the minimal modulus that the Artin symbol $\big(\frac{\cdot}{L/K}\big)$ admits: it is a product of the ramified primes, appearing to exponents that may be computed using ramification groups.
\begin{lem}[\cite{Nakagawa}, Lemma 1.3; \cite{Hasse}] \label{lem:conductor}
The conductor of the extension $K_6/K_2$ is the principal ideal $(d) \subseteq \OO_{K_2}$.
\end{lem}
\begin{proof}
First, we are asserting that $K_6$ is unramified at infinity, which is automatic for a Galois extension of odd degree.

So fix a finite prime $\pp_2$ of $K_2$ lying above some prime $p \in \ZZ$. We would like to prove that the exponent $n_{\pp_2}$ of $\pp_2$ in the conductor is equal to
\[
  v_{\pp_2}(d) = \frac{1}{2}(v_{\pp_2}(\Delta) - v_{\pp_2}(\Delta_0)).
\]
Recall that, by class field theory, $n_{\pp_2}$ is the least nonnegative integer such that the upper ramification group $G^i(K_6/K_2, \pp_2)$ vanishes. Now since $K_6/K_2$ is of prime order $3$, the sequence of ramification groups must be of the simple form
\[
  G_i \cong \begin{cases}
    \ZZ/3\ZZ & \text{for $i < n$, where $n \geq 0$ is some integer} \\
    0 & \text{otherwise.}
  \end{cases}
\]
This implies that the upper ramification groups $G^i$ are exactly the same as the lower ones, and thus $n_{\pp_2} = n$. On the other hand, the ramification groups are connected with the different, $\dd_{K_6/K_2},$ by
\[
  v_{\pp_6}(\dd_{K_6/K_2}) = \sum_{i \geq 0} (|G_i| - 1) = 2n
\]
where $\pp_6$ is a prime of $K_6$ lying over $\pp_2$. Accordingly, it suffices to prove the identity
\begin{equation}\label{eq:CFT dd and disc}
  v_{\pp_6}(\dd_{K_6/K_2}) = v_{\pp_2}(\Delta) - v_{\pp_2}(\Delta_0).
\end{equation}
We consider the various ways that $p$ can ramify in $K_3$.
\begin{enumerate}
  \item\label{it:p unram} If $p$ is unramified in $K_3$, then $p \nmid \Delta$, so $p \nmid \Delta_0$ and $p$ is unramified in $K_6$. Thus \eqref{eq:CFT dd and disc} holds, as every term is $0$.
  \item\label{it:p part ram} Suppose $p$ is partially ramified in $K_3$, that is, $p = \pp_3^2\pp_3'$, where $\pp_3$ and $\pp_3'$ are distinct primes of $K_3$. Then $\pp_3$ must split and $\pp_3' = \pp_6^2$ must ramify in $K_6$. In particular, $K_6/K_2$ is unramified at $\pp_2$ (otherwise the ramification index of $\pp_6$ over $p$ would be divisible by $3$), so the left side of \eqref{eq:CFT dd and disc} is zero. Now consider the completed algebra
\[
  (K_3)_p = K_3 \tensor_\ZZ \ZZ_p.
\]
In view of the splitting of $p$, we must have $(K_3)_p \cong \ZZ_p \cross \Gamma$, where $\Gamma/\ZZ_p$ is a ramified quadratic extension. The discriminant of $\Gamma$ (which is well-defined up to multiplication by $(\ZZ_p^\cross)^2$) is
\[
  \disc \Gamma = \disc (K_3)_p = \disc K_3 = \Delta.
\]
This implies that $\Delta$ is a fundamental discriminant over $\ZZ_p$, that is to say $p^2 \nmid \Delta$ (if $p \neq 2$) or $\Delta \not\equiv 0,4$ mod $16$ (if $p = 2$). But $\Delta_0$ is also a fundamental discriminant over $\ZZ_p$, and $\Delta = \Delta_0 d^2$. We thus get $p \nmid d$ and $v_p(\Delta) = v_p(\Delta_0)$, proving \eqref{eq:CFT dd and disc}.
  \item\label{it:p tot ram} We are left with the case that $p = \pp_3^3$ is totally ramified in $K_3$. Note that $\pp_2$ must be totally ramified in $K_6$, as otherwise the ramification index of $\pp$ in $K_6$ would be at most $2$. Then the quadratic extensions $K_6/K_3$ and $K_2/\QQ$ must be of the same type at $\pp_3$ and $p$ respectively (both split, both inert, or both ramified). If both are unramified, then we easily get
  \begin{align*}
      v_{\pp_6}(\dd_{K_6/K_2})
    &= v_{\pp_6}(\dd_{K_6/\QQ})
    = v_{\pp_6}(\dd_{K_3/\QQ})
    = v_{\pp_3}(\dd_{K_3/\QQ})
    = v_{p}(\Delta)
    = v_{\pp_2}(\Delta) \\
    &= v_{\pp_2}(\Delta) - v_{\pp_2}(\Delta_0),
  \end{align*}
  as desired. This leaves the case where all the extensions in \eqref{eq:4 fields} are totally ramified. If $p \neq 2$, we apply Dedekind's theorem on the different, which states that for $L/K$ a totally tamely ramified degree-$e$ extension of local fields, $v_{L}(\dd_{L/K}) = e-1$:
  \begin{align*}
    v_{\pp_6}(\dd_{K_6/K_2})
    &= v_{\pp_6}(\dd_{K_6/K_3}) + v_{\pp_6}(\dd_{K_3/\QQ}) - v_{\pp_6}(\dd_{K_2/\QQ}) \\
    &= v_{\pp_6}(\dd_{K_6/K_3}) + 2v_{\pp_3}(\dd_{K_3/\QQ}) - 3v_{\pp_2}(\dd_{K_2/\QQ}) \\
    &= 1 + 2v_p(\Delta) - 3 \\
    &= 2v_p(\Delta) - 2 \\
    &= 2v_p(\Delta) - 2v_p(\Delta_0) \\
    &= v_{\pp_2}(\Delta) - v_{\pp_2}(\Delta_0).
  \end{align*}
  If $p = 2$, approaching the proof along similar lines leads to difficulties. In fact, this case cannot occur. The totally ramified extension of local fields $(K_3)_2/\QQ_2$ must have a primitive element satisfying an Eisenstein polynomial $x^3 + bx^2 + cx + d$, where $b,c,d$ are even and $4 \nmid d$. Computing the discriminant of $(K_3)_2$ via \eqref{eq:discriminant}, we notice that all terms are divisible by $16$ except for $-27a^2d^2$, which equals $4$ mod $16$. Accordingly $\Delta \equiv 4$ mod $16$ and $\Delta_0 \equiv 1$ mod $4$, contradicting the supposition that $2$ is ramified in $K_2$. \qedhere
\end{enumerate}
\end{proof}

Thus we have associated to each non-Galois cubic field extension of discriminant $\Delta = \Delta_0 m^2$ an Artin map
\begin{equation} \label{eq:Artin}
  \chi : I_{K_2}(m)/I_{K_2}(m,1) \twoheadrightarrow \mu_3
\end{equation}
from the ray class group mod $m$ onto a cyclic group of order $3$, uniquely defined up to sign. (Here, as usual, $I_{K_2}(m)$ denotes the ideals prime to $m$ and $I_{K_2}(m,1)$ the principal ideals generated by elements congruent to $1$ mod $(m)$.)

Conversely, given such a map $\chi$, class field theory gives a cyclic extension $K_6/K_2$ of conductor dividing $m$ of which it is the Artin map. However, not all of these extensions $K_6$ will be $S_3$-Galois over $\QQ$. The maps we want are those such that applying the nontrivial automorphism $x \mapsto \bar{x}$ of $\Gal(K_2/\QQ)$ interchanges the two nonzero elements of $\mu_3$, that is,
\begin{equation}\label{eq:psi antisymmetry}
  \chi(\bar{\aa}) = \chi(\aa)^{-1}.
\end{equation}
Then, by the uniqueness part of the Existence Theorem of class field theory, $K_6$ has an automorphism $\tau$ such that $\tau \sigma \tau^{-1} = \sigma^{-1}$ for all $\sigma \in \Gal(K_6/K_2)$; in other words, $K_6$ is $S_3$-Galois over $\QQ$. (If we imposed instead the condition $\chi(\bar{\aa}) = \chi(\aa)$, we would instead pick out $\ZZ/6\ZZ$-Galois fields.) Clearly \eqref{eq:psi antisymmetry} implies that $\chi$ vanishes on integers. Moreover, the converse is true: If $\aa \subseteq K_2$ is any integral ideal prime to $m$, we get
\[
  \chi(\bar{\aa}) = \chi(N(\aa))/\chi(\aa) = \chi(\aa)^{-1}.
\]
So we are only seeking Artin maps that factor through the quotient
\[
  I_{K_2}(m)/I_{K_2}(m,\ZZ)
\]
where $I_{K_2}(m,\ZZ)$ is the subgroup of principal ideals generated by an element congruent to some integer (necessarily coprime to $m$) modulo $(m)$. This is a familiar quotient group: it is the ring class group of the quadratic order $\OO_{\Delta} = \ZZ + m\OO_{\Delta_0}$ (\cite{Nakagawa}, Lemma 1.9).

Any $\chi : \Pic \OO_{\Delta} \twoheadrightarrow \mu_3$ yields an $S_3$-Galois field $K_6/\QQ$, and hence a non-Galois cubic field $K_3/\QQ$. The discriminant of $K_3$ will be $\Delta = \Delta_0 m^2$ unless $\chi$ vanishes on a larger subgroup $I_{K_2}(d,\ZZ) \intsec I_{K_2}(m)$, in which case $\chi$ has conductor $d$ (for the smallest such $d$) and $\Disc(K_3) = \Delta_0 d^2$.

Say that an integer $x$ \emph{squarely divides} an integer $y$ if $y/x$ is the square of an integer. We have just proved:

\begin{lem}\label{lem:fields and Pic}
If $\Delta = \Delta_0 m^2$ is a non-square integer, the Artin map provides a bijection between cubic fields whose discriminant squarely divides $\Delta$ and group epimorphisms
\[
  \chi : \Pic(\OO_\Delta) \twoheadrightarrow \mu_3
\]
up to sign.
\end{lem}

\begin{rem}
In particular, we have shown that for any such $\chi$, the conductor $\cond(\chi)$ is a principal ideal generated by an integer. An elementary proof of this fact is also possible; the details are left to the reader.
\end{rem}

The case that $\Delta = m^2$ is a square, that is, $\Delta_0 = 1$, is similar but simpler, as we need only apply class field theory to the Galois extension $K_3/\QQ$ itself. The method of Lemma \ref{lem:conductor} shows that $K_3/\QQ$ has conductor $m$, yielding an Artin map
\[
  \chi_1 : I_{\QQ}(m)/I_{\QQ}(m,1) \twoheadrightarrow \mu_3.
\]
In the interest of conformity with the preceding, we use the bijection $\chi_1 \mapsto (\chi_1,\chi_1^{-1})$  to put these in bijection with maps
\[
  \chi : I_{\QQ \cross \QQ}(m) / I_{\QQ \cross \QQ}(m,\ZZ) = \Pic(\OO_{m^2}) \to \mu_3,
\]
yielding the following uniform parametrization.

\begin{lem}\label{lem:fields and Pic, genl}
Let $\Delta \in \Discs$, and let $\OO_\Delta$ denote the quadratic ring of discriminant $\Delta$. The Artin map provides a bijection between cubic fields whose discriminant squarely divides $\Delta$ and group epimorphisms
\[
  \chi : \Pic(\OO_\Delta) \twoheadrightarrow \mu_3
\]
up to sign.
\end{lem}

\subsection{Splitting types}
Suppose we wish to compute $h(\Delta)$ for some $\Delta = \Delta_0 m^2$. We can list all of the fields $K$ whose discriminant $\Delta_0 d^2$ squarely divides $\Delta$; then we must count orders in $K$ of index $m/d$. By Remark \ref{rem:subrings}, we can compute this knowing the splitting types of $K$ at each of the primes dividing $m/d$. The following proposition (which should also be credited to Hasse: see \cite{Hasse}, p{.} 568) gives a simple way to find these splitting types in terms of the corresponding Artin map $\chi$.

\begin{prop} \label{prop:spl type}
Let $\chi : \Pic(\OO_{\Delta_0 d^2}) \to \mu_3$ be a primitive character (i.e{.} one that does not factor through any $\Pic(\OO_{\Delta_0 d'^2})$, $d' | d$). Let $p \in \ZZ$ be a prime. The splitting type of $p$ in the $\QQ$-algebra $K$ corresponding via Theorem \ref{thm:LHS} to $\chi$ is
\begin{itemize}
  \item $1^3$ if $p|d$,
  \item $1^2 1$ if $p \nmid d$ but $p | \Delta_0$,
  \item $12$ if $p \nmid d$ and $p$ is inert in $\OO_{\Delta_0}$ (i.e{.} the Kronecker symbol $\left(\dfrac{\Delta_0 d^2}{p}\right)$ has the value $-1$),
  \item $1^3$ if $p \nmid d$ and $p = \pp \bar{\pp}$ in $\OO_{\Delta_0 d^2}$ with $\chi(\pp) = 1$;
  \item $3$ if $p \nmid d$ and $p = \pp \bar{\pp}$ in $\OO_{\Delta_0 d^2}$ with $\chi(\pp) \neq 1$.
\end{itemize}
In particular, all splitting types can be told apart merely by reference to the discriminant $\Delta = \Delta_0 d^2$, except for $111$ and $3$.
\end{prop}
\begin{proof}
If $\Delta_0 = 1$, then $K$, being Galois, can only have splitting type $111$, $3$, or $1^3$, and it is easy to see that these cases occur exactly in the cases claimed. So assume that $\Delta_0 > 1$.

The primitivity of $\chi$ implies that
\[
  \Disc K = \Delta = \Delta_0 d^2.
\]
We immediately see that $K$ is ramified if and only if $p | \Delta$. If $K$ has splitting type $1^2 1$, then we are in case \ref{it:p part ram} of Lemma \ref{lem:conductor}, and we see that $p \nmid d$ (and hence that $p|\Delta_0$). If $K$ has splitting type $1^3$, then we are in case \ref{it:p tot ram} of Lemma \ref{lem:conductor}. By \eqref{eq:CFT dd and disc}, we have
\[
  2v_{\pp_2}(d) = v_{\pp_2}(\Delta) - v_{\pp_2}(\Delta_0) = v_{\pp_6}(\dd_{K_6/K_2}) > 0,
\]
since $K_6/K_2$ is totally ramified at $\pp_2$, and thus $p|d$.

In the case that $p$ is unramified, there are just three cases: the splitting types $111$, $12$, $3$ are also the cycle types of $\Frob_p$ as an element of $\Gal(K_6/\QQ) \cong S_3$. Note that cycle type $12$, being the only odd permutation, corresponds exactly to the case that the discriminant field $\QQ(\sqrt{\Delta})$ is inert at $p$. The other two cases can be told apart via class field theory: here $p$ splits as a product $p = \pp_2\bar{\pp_2}$ in $\OO_{\Delta_0}$ and hence as a product $\pp\bar\pp$ in $\OO_\Delta$ (as $p\nmid d$). The Artin symbol $\chi(\pp) = \chi(\pp_2)$ vanishes if and only if $\pp_2$ splits completely in $K_6$, which happens exactly when $p$ splits completely in $K_3$.
\end{proof}

In our computation of $h(\Delta)$, we are still missing the contribution of the nondomains, which are subrings of $K_r = \QQ \cross \OO_{\Delta_0}$ of index $m$. The splitting type of $K_r$ at every prime $p$ is either $111$, $12$, or $1^2 1$, and one finds that applying Proposition \ref{prop:spl type} to the trivial character $\chi = 1$, of conductor $d = 1$, yields the right answer. So it makes sense to define the Artin map of a nondomain to be identically $1$.

Note that $\QQ \cross \OO_{\Delta_0}$ has twice as many automorphisms as the fields whose discriminants squarely divide $\Delta$ ($6$ if $\Delta_0 = 1$, and $2$ otherwise). On the other hand, if we sum up over all maps $\chi : \Pic \OO_\Delta \to \mu_3$, then the Artin maps corresponding to fields $K$ appear twice, due to the sign ambiguity, but the trivial Artin map appears only once. So, counting the automorphisms carefully, we arrive at the following result.

\begin{thm}\label{thm:LHS}
Let $\Delta = \Delta_0 m^2 \in \Discs$. Let $w_\Delta = 3$ if $\Delta$ is a square, $1$ otherwise. The following quantities are equal:
\begin{itemize}
  \item $2 w_\Delta h(\Delta)$;
  \item The sum, over all characters $\chi : \Pic \OO_\Delta \to \mu_3$, of the number of subrings of index $m/\cond \chi$ in the cubic ring $C$ whose local splitting types are determined by $\Disc C = \Delta_0 (\cond \chi)^2$ and $\chi$.
\end{itemize}

\end{thm}

\section{Finishing the proof of Theorem \ref{thm:main}}
\label{sec:finish}
To finish the proof for a value $\Delta = \Delta_0 m^2$, we would like to equate the expression for $2 w_\Delta h(\Delta)$ in Theorem \ref{thm:LHS} with the expression for $2 w_\Delta \eta_\Delta \hat h(\Delta)$ in Theorem \ref{thm:RHS}, which is reproduced here for convenience:

\begin{itemize}
  \item $2 w_\Delta \eta_\Delta \hat h(\Delta)$ is the number of invertible ideals $I$ of norm $g$ whose class is a cube in orders $\OO_{\Delta'}$ for integers $g>0$, $\Delta'$ satisfying $\Delta'g^2 = \Delta$, each counted with weight
  \[
    \lvert\Pic(\OO_{\Delta'})[3]\rvert.
  \]
\end{itemize}

It is not hard to turn $2 w_\Delta \eta_\Delta \hat h(\Delta)$ into a character sum, as follows. If $I$ is an invertible ideal in $\OO_{\Delta'}$, then
\[
  \sum_{\chi \in \Hom(\Pic \OO_{\Delta'}, \mu_3)} \chi(I) =
  \begin{cases}
    \lvert\Hom(\Pic \OO_{\Delta'}, \mu_3)\rvert
    = \lvert\Pic(\OO_{\Delta'})[3]\rvert &
    \text{if $I$ is of cubical class} \\
    0 & \text{otherwise.}
  \end{cases}
\]
So
\begin{align*}
  w_\Delta \eta_\Delta \hat h(\Delta) &= \sum_{c f = m}
  \sum_{\substack{I \subseteq \OO_{\Delta_0 c^2} \\ \text{invertible, norm $f$}}}
  \sum_{\chi : \Pic \OO_{\Delta_0 c^2} \to \mu_3} \chi(I) \\
  &= \sum_{\chi : \Pic \OO_{\Delta} \to \mu_3}
  \sum_{\substack{c f = m, \\ \cond(\chi) | c}}
  \sum_{\substack{I \subseteq \OO_{\Delta_0 c^2} \\ \text{invertible, norm $f$}}} \chi(I).
\end{align*}

It suffices to prove that, at least for cubefree $m$ (in view of Remark \ref{rem:recursion}), the contribution of each $\chi$ to $2 w_\Delta h(\Delta)$ and $2 w_\Delta \eta_\Delta \hat h(\Delta)$ is the same. In other words, fix a $\chi$; let its conductor be $c_1$, and let $\Delta_1 = \Delta_0c_1^2$, $m_1 = m/c_1$. We will prove that the number of subrings of the corresponding $\QQ$-algebra $K_\chi$ of index $m_1$ is equal to
\begin{equation} \label{eq:chi term}
  \sum_{c' f = m_1}
  \sum_{\substack{I \subseteq \OO_{\Delta_1 c'^2} \\ \text{invertible, norm $f$}}} \chi(I).
\end{equation}

We first observe that the number of subrings is a multiplicative function of $m_1$ and claim that \eqref{eq:chi term} is also. If $m_1 = m_2m_3$ with $\gcd(m_2,m_3) = 1$, then we get corresponding decompositions $c' = c'_2c'_3$ and $f = f_2f_3$. An invertible ideal $I$ of norm $f_2f_3$ in $\OO_{\Delta_0(c_0c'_2c'_3)^2}$ can be decomposed uniquely as a product $I_2I_3$, where $I_i$ is an invertible ideal of norm $f_i$; since $f_i$ is prime to $c'_{5-i}$, invertible ideals of norm $f_i$ in the orders $\OO_{\Delta_1(c'_2c'_3)^2}$ and $\OO_{\Delta_1{c'_i}^2}$ are in bijection.

Thus we can assume that $m_1 = p^k$ is a prime power. We once again have a local problem. There are several cases. The case $k=0$ is trivial, so we have $k=1$ or $k=2$. The following table shows the types of invertible ideals on which we must evaluate $\chi$ and sum:
\[
\begin{tabular}{c|ccc}
 & $c' = 1$ & $c' = p$ & $c' = p^2$ \\ \hline
 $k=1$ & norm $p$ in $\OO_{\Delta_1}$ & unit ideal in $\OO_{\Delta_1p^2}$ & \\
 $k=2$ & norm $p^2$ in $\OO_{\Delta_1}$ & [norm $p$ in $\OO_{\Delta_1p^2}$] & unit ideal in $\OO_{\Delta_1p^4}$
\end{tabular}
\]
The bottom middle entry has been placed in brackets because no such ideals exist. Suppose to the contrary that we had a map $\phi : \OO_{\Delta_1p^2} \to \FF_p$ with kernel an invertible ideal. Let $\xi$ be a generator of $\OO_{\Delta_1}$, so $\OO_{\Delta_1p^2} = \ZZ[p\xi]$. We have $\phi(p\xi)^2 = p \cdot \phi(p\xi^2) = 0$, so $\phi(p\xi) = 0$ and hence $\ker \phi = \ZZ\<p, p\xi\> = p \OO_{\Delta_1}$, which is not an invertible ideal.

If $p|c_1$, then $k=1$, and by the same argument, $\OO_{\Delta_1}$ has no invertible ideals of index $p$, so the value of \eqref{eq:chi term} is simply $1$, coming from the unit ideal in $\OO_{\Delta_1p^2}$. This accords with the number $1$ of subrings of index $p$ in a maximal ring of splitting type $1^3$, as tabulated in \eqref{eq:subring table}.

In the remaining cases, $p\nmid c_1$, so $p$ has the same splitting type in $\OO_{\Delta_1}$ as in $\OO_{\Delta_0}$. We only have to sum $\chi$ over ideals of norm $p^k$ in $\OO_{\Delta_1}$, all of which will be invertible, and add the contribution $1$ coming from the unit ideal in $\OO_{\Delta_1p^{2k}}$. 

If $p$ is inert, then $\OO_{\Delta_1}$ has no ideals of norm $p$ and one ideal of norm $p^2$, namely $(p)$, with $\chi((p)) = 1$. So the total \eqref{eq:chi term} is $1$ for $k = 1$ and $2$ for $k = 2$, in accordance with \eqref{eq:subring table} for $K_\chi$ having splitting type $12$.

If $p = \pp^2$ ramifies in $\OO_{\Delta_1}$, then $\OO_{\Delta}$ has one ideal each of norm $p$ and $p^2$. Note that $\chi(\pp) = 1$ since $\pp^2 = (p)$ is principal. So the total \eqref{eq:chi term} is $2$ for both $k=1$ and $k=2$, in accordance with \eqref{eq:subring table} for splitting type $1^21$.

Finally, if $p = \pp\bar\pp$ is split in $\OO_{\Delta_1}$, then $\OO_{\Delta_1}$ has two ideals of norm $p$ ($\pp$ and $\bar{\pp}$) and three ideals of norm $p^2$ ($\pp^2$, $\bar\pp^2$, and $\pp\bar\pp = (p)$). We know that $\chi(\bar\pp) = \chi(\pp)^{-1}$. Adding up $\chi$ on the relevant ideals in the two cases $\chi(\pp) = 1$, $\chi(\pp) \neq 1$ matches the four entries of \eqref{eq:subring table} for splitting types $111$ and $3$, finishing the proof.
\qed

\section{Acknowledgements}
I thank Manjul Bhargava for communicating this beautiful problem to me on a visit to Princeton (during which Bhargava was supposed to be presenting at Harvard, but his flight was felicitously canceled). I thank my fellow Cambridge students for encouraging me to propose the problem as an essay topic and write a Part III essay on it. I thank Jack Thorne for reading it as essay assessor.

\bibliography{../Master}
\bibliographystyle{plain}

\end{document}